\documentclass{amsart}
\usepackage[T1]{fontenc}

\newtheorem{thm}{Theorem}[section]
\newtheorem{lem}[thm]{Lemma}
\newtheorem{prop}[thm]{Proposition}
\theoremstyle{definition}
\newtheorem{defn}[thm]{Definition}

\numberwithin{equation}{section}

\usepackage{latexsym,amssymb,amsmath}

\newcommand{\C}{\mathbb C}
\newcommand{\Z}{\mathbb Z}
\newcommand{\N}{\mathbb N}
\newcommand{\R}{\mathbb R}
\newcommand{\Q}{\mathbb Q}
\newcommand{\F}{\mathbb F}

\newcommand{\sma}{\left(\begin{array}}
\newcommand{\fma}{\end{array}\right)}
\newcommand{\injects}{\hookrightarrow}

\newtheorem{ex}[thm]{Example}
\newtheorem{co}[thm]{Corollary}

\begin{document}

\title[Groups acting on products of locally finite trees]%
{Groups acting on products of locally finite trees}

\author{J.O.Button}
\address{Selwyn College; University of Cambridge;
Cambridge CB3 9DQ;
U.K.}
\email{j.o.button@cam.ac.uk}



\begin{abstract}
We examine the question of which finitely generated groups act properly
on a finite product of locally finite simplicial trees
and present
evidence in favour of hyperbolic surface groups having such an action.
We also give a completely explicit embedding of the genus 2 closed hyperbolic
surface group in $SL_2(\F_p(x,y))$ for any prime $p$.
\end{abstract}

\maketitle

\section{\bf Introduction}

A useful technique which can be used to understand a finitely generated group
$G$ is to find a geodesic metric space $X$ on which it acts
geometrically, that is properly and cocompactly by isometries. Of course
every $G$ has such an action on its Cayley graph but
if the space $X$ is well behaved geometrically, for instance if it is a
hyperbolic space or a CAT(0) space, then this allows us to establish
properties for the group $G$.

However it can be quite restrictive asking for the action to be geometric.
Indeed a group $G$ admits a
geometric action on some simply connected geodesic metric space if and
only if $G$ is finitely presented (\cite{bh} I.8.11),
so if $G$ has a geometric action on such a space but $G$ also has a subgroup
$H$ which is finitely generated but
not finitely presented then there is no geometric action of $H$ on any
simply connected geodesic metric space.

One way around this is to allow actions of a group $G$
by isometries which are
proper, but not necessarily cocompact,
on a geodesic metric space $X$. This has the big advantage that every 
subgroup of $G$ acts properly too. On the other hand, this might be
too general:
for instance it is an open question as to whether every countable group
has a proper action on some CAT(0) space. However in this case it is
possible to strengthen the hypotheses on the action without the full force
of a geometric action, since the class of groups having a proper and
semisimple action on a CAT(0) space is an interesting class with good
geometric properties (see for instance \cite{bh} III.$\Gamma$.1.1).

If we think of finite generation/presentability as finiteness conditions
on the group and acting properly/cocompactly as finiteness conditions
on the action, we might wonder about finiteness conditions on the space
$X$. Hence we could also insist that $X$ is a proper metric space (closed
balls are compact). This has the advantage that acting properly and
the stronger notion of acting metrically properly are then equivalent.
We could further insist that $X$ has some form of non-positive or negative
curvature, for instance $X$ is hyperbolic.
But every countable group acts properly on a locally finite
(and hence proper) hyperbolic graph. Alternatively we could specialise
to the case where $X$ is a tree, where here tree means simplicial tree
throughout and group actions are always by simplicial automorphisms
but they may invert edges. However if $G$ is finitely generated then
$G$ acts properly on some tree if and only if $G$ acts properly on some
tree of bounded valence if and only if $G$ is virtually free, thus the class
of groups so obtained is very limited.

Another reason to vary the type of spaces we are considering is that
we would like the notion of acting properly on a space in a given class
to be not just preserved by subgroups but also by finite index supergroups.
With geometric actions this is far from clear, for instance if $H$ is a
CAT(0) group and it has finite index $i$ in $G$, it is unknown whether $G$
must also be CAT(0). However if we merely require a proper action then
as long as our class of spaces is closed under taking direct products
(say with the $\ell_2$ product metric), we can induce an action of $G$
on $X^i$ from an action of $H$ on $X$ and here properness will be
preserved.
Examples of such classes $\mathcal C$ where the question of which groups
act properly on spaces in $\mathcal C$ has been investigated recently
include the class of CAT(0) cube complexes, where one might or might
not restrict to finite dimensional and/or locally finite spaces, and
finite products of quasi-trees.

In this paper our focus is on groups that
act properly (by automorphisms)
on a finite product of trees, where each tree is equipped with
its path metric and the space is equipped with the Euclidean product metric.
This begs the question: should we consider arbitrary trees (in which case
should we further consider proper or metrically proper actions), or only
locally finite trees? In the first case a wide range of groups is known
to have such an action (under either definition of proper) but for locally
finite trees the known examples are much more restricted. Therefore our
emphasis here will be on the locally finite case and seeing how it differs
from the case of arbitrary trees.

We provide the basic background in Section 2 and Lemma \ref{equiv} shows that,
at least for finitely generated groups, it does not matter if we consider
products of locally finite or of bounded valence trees. But the question
of which groups possess proper actions on finite products of locally finite
trees seems quite mysterious. For instance we have virtually free groups and
Burger - Mozes groups (these two examples also act cocompactly), along
with subgroups and direct products of these examples. In 2.3 we look at
finitely generated groups containing an infinite but locally finite subgroup,
such as lamplighter groups $F\wr\Z$ for $F$ finite which also possess such an
action. We also examine the Houghton groups which do not possess any
proper action on a finite product of locally finite trees, although they
do once the locally finite condition is removed.

For any finite non-trivial group $F$, we show in Section 3 that, as opposed
to $F\wr\Z$, the wreath product $F\wr\Z^2$ does not have any proper action
on a finite product of locally finite trees. This result sits alongside
that of \cite{gntr} Proposition 1.3, which shows that $F\wr H$
cannot act properly on any finite dimensional CAT(0) cube complex (which
includes any finite product of trees) if $H$ contains a non-abelian free
group.

But what about any example of a word hyperbolic group with such an action
but which is not virtually free? The obvious case to try would be the
fundamental group
$S_g$ of a closed orientable hyperbolic surface $\Sigma_g$ with genus 
$\geq 2$. In fact exactly this question was raised in \cite{flss} which
gives partial results and which was the motivation for much of this work.
Here we present evidence in favour of the existence of a proper action
for our surface groups $S_g$, by
considering how these groups can act on a single locally finite tree.
In Section 4, Theorem \ref{injcloc} shows that a necessary condition for
such an action by $S_g$ (or indeed any finitely generated, torsion free
group which does not contain $\Z\times\Z$)
is that for any non-identity element $\gamma$,
there is an action of $S_g$ on some locally finite tree which is
minimal, faithful and such that $\gamma$ acts as a loxodromic element.
Note that without the faithful condition we could use the residual
freeness of $S_g$ to obtain a suitable action on the Cayley graph
of a finitely generated free group, whereas without the minimal
condition we could take the same action and convert it into a
faithful action on a different locally finite tree by ``decorating'' the
vertices of this tree.

It is pointed out in \cite{flss} that if we could find an embedding
of $S_g$ in $SL(2,\F)$ for $\F$ a global field of positive characteristic
(a finite degree field extension of $\F_p(x)$ for some prime $p$) then
$S_g$ would act properly on a finite product of locally finite trees, by
taking the Bruhat - Tits trees associated to a finite number of
discrete valuations $v$ on $\F$. Given such a global field $\F$ and 
some $v$, we have the local field $k_v$ obtained by taking
the completion of $\F$ with respect to $v$. Now it is clear
that for any prime $p$, the group $S_g$ embeds in $SL(2,k)$ for some
local field $k$ of characteristic $p$. For instance, using \cite{lubseg}
Window 8 Theorem 1 which was originally due to 
Malceev, the two facts that a finite
rank free group has a 2 dimensional faithful linear representation
in characteristic $p$ and
that $S_g$ is fully residually free implies that $S_g$ also has such a
representation, say over the field $K$. As $S_g$ is finitely generated,
we can take $K$ to be finitely generated over its prime subfield $\F_p$
which means that $K$ is a finite extension of
$\F_p(t_1,\ldots ,t_d)$ where obviously $d>0$ and $t_1,\ldots ,t_d$
are algebraically independent elements. Thus if $k$ is any
local field containing $\F_p(x)$, say $\F_p((x))$, and we set $t_1=x$
then the uncountability of $k$ means that the field
$\F_p(t_1,\ldots ,t_d)$ will embed in $k$ and moreover our finite
extension $K$ of $\F_p(t_1,\ldots ,t_d)$ will embed in some finite extension
of $k$ which will also be a local field.
 
Thus we can view the question of whether $S_g$ embeds in $SL(2,\F)$ for
$\F$ a global field (which implies the existence of  a proper action)
as the top question in a series of questions as to how ``economical''
we can take our field $K$ to be, since this is the $d=1$ case. As this
is unknown, we can instead ask how small we can take $d$ in a faithful
representation of $S_g$ and also how small the degree of our field extension
needs to be. In \cite{flss} it was shown that 
for every prime $p$ at least 5, there is a faithful embedding
of $S_2$ (and hence $S_g$) in $PGL(2,K)$ where $K$ is a finite extension of 
$\F_p(x,y)$, thus we can take $d=2$ if $p\neq 2,3$. In Section 5 we remove the
condition on $p$ and the need for a finite extension in that we provide in
Theorem \ref{mat} and Corollary \ref{mat2} a completely explicit
embedding of $S_2$ (and hence $S_g$ for any $g\geq 2$) in
$SL(2,\F_p(x,y))$ for any prime $p$. (In fact, as is presumably usual, the
four matrices provided in Theorem \ref{mat} are independent of $p$ and
work for every odd prime $p$, whereas $p=2$ requires a different
representation which is given in Corollary \ref{mat2}.) As this is the
most ``economical'' field possible for a 2 dimensional representation of
$S_g$ in positive characteristic short of actually establishing
a proper action on a finite product of locally finite trees,
we claim that this provides evidence in favour of
the existence of such an action. Moreover our faithful representation
of $S_g$ in $SL(2,\F_p(x,y))$ provides in Corollary \ref{moreacc}, for 
any finite collection of non-identity elements in $S_g$, a minimal faithful
action of $S_g$ on a locally finite tree in which all of those
elements act loxodromically,
thus providing further evidence in favour.

Our faithful representation of $S_2$ is obtained by using results in
\cite{con} on how to determine whether a pair of elements in the
automorphism group of a locally finite tree generate a discrete and
faithful copy of $F_2$. We then combine this with a result in \cite{sha}
on constructing faithful linear representations of an amalgamated free
product, given faithful linear representations of the factor groups.

Acknowledgements: The author would like to thank the anonymous referee
for many enlightening comments and suggestions on an earlier draft
of this article, including Theorem \ref{ref} and the examples involving
related groups.
 
\section{\bf Actions on locally finite trees}

\subsection{Actions on a single tree}

For us, all trees in this paper are simplicial trees, defined 
in the standard combinatorial way as in \cite{ser}, but then
regarded as metric spaces by equipping them with the resulting path
metric where each edge has length 1. Given any tree
$T$, we will use $Aut(T)$ to denote
the group of simplicial automorphisms of $T$ and these will be 
isometries of $T$. Moreover, saying that an abstract group $G$ acts on a tree 
$T$ will also mean that $G$ acts on $T$ by
simplicial automorphisms, but it needs not imply that the action is
faithful.
We will allow our actions to have edge inversions, since
if $G$ does invert an edge when
acting suitably on some tree $T$ then it would act suitably without edge 
inversions 
and still by automorphisms on the barycentric subdivision of $T$.
With this in mind, we say that an action of a group $G$ on a tree $T$
is free if no group element
except the identity has a fixed point in $T$ (equivalently no vertex
or midpoint of an edge in $T$ is fixed, or again equivalently no
vertex of the barycentric subdivision is fixed). In this case the action
of $G$ will of course be faithful, but in general we do not suppose this.

As $T$ will be a geodesic hyperbolic space and $G$ will act on $T$ by
isometries, we know that this action will fall into one of five distinct
types: it could have bounded orbits, be parabolic/horocyclic, be lineal,
be quasi-parabolic/focal, or be of general type in which case $G$ must
contain a copy of the non-abelian free group $F_2$. In the last three cases
$G$ must contain a loxodromic element whereas $G$ will not in the first
two cases. As we are acting on a tree,
there are further restrictions: first all elements must either act
loxodromically or elliptically. Also having bounded orbits will imply
the action has a global fixed point and (by a result of Serre in \cite{ser})
parabolic/horocyclic actions can only occur if $G$ is not finitely generated.

Given any subset $S$ of the vertices of a simplicial tree $T$, we can form
the convex closure $CCl(S)$ of $S$ in $T$ which is a subtree of $S$.
\begin{defn} \label{cre} Given
  a group $G$ acting on a tree $T$ containing a loxodromic element,
  the {\bf core} of this action $C_G(T)$ is the unique minimal subtree of $T$
  invariant under $G$, which exists and is equal to the union of all axes
  of loxodromic elements. 
\end{defn}
Note that (by minimality and convexity) $C_G(T)$ is equal to the convex
closure of the orbit of any vertex contained in $C_G(T)$. Moreover if $G$ is
finitely generated then the quotient $G\backslash C_G(T)$
is a finite graph (see for instance \cite{dunw} Proposition I.4.13).

We are interested in groups that have an interesting simplicial action on
a tree and particularly on some locally finite or bounded valence
tree, but let us first consider groups $G$ without such actions. Recall
that $G$ has Serre's property ($FA$) if every simplicial action of $G$
on a tree has a global fixed point. By \cite{ser} I.6 Theorem 15
if $G$ is finitely generated then this is equivalent to
saying that $G$ does not split as an HNN extension or amalgamated free
product. However if $G$ is countable but not finitely generated then $G$
never has ($FA$) because we can
use the coset construction as described in \cite{ser} I.6.1 or in  \cite{bh}
II.7.11. Here we provide more detail as it will be useful
elsewhere in this paper. The basic idea is that we can express $G$ as an
increasing union $\cup_{i=0}^\infty G_i$ of subgroups $G_i$ with $G_0=\{e\}$
and with $G_i$ properly contained in $G_{i+1}$ (and if we can do this
then $G$ must be infinitely generated).
We now form a tree $T$ with vertices partitioned into levels
$0,1,2,\ldots$. The 0th level vertices are the elements of $G$, the 1st level
vertices are the cosets of $G_1$ in $G$ and similarly for the $i$th
level vertices. Each edge of $T$
(which can be thought of as directed) runs from
some vertex on level $i$ to a vertex on level $i+1$, where for any $g\in G$
the coset $gG_i$ and the coset $gG_{i+1}$ are joined by an edge.
As $G$ acts by left multiplication on the cosets in $G$
of any subgroup $H$, this provides an action of $G$ on the vertices of $T$
which is by tree automorphisms and such that $T$ has one end fixed by $G$.
The quotient graph $G\backslash T$ is a ray with one vertex for each level,
with the stabiliser of a vertex $gG_i$ on level $i$ equal to
the subgroup $gG_ig^{-1}$, so note that if $g$ is in $G_j$ then this
stabiliser is contained in $G_{\max(i,j)}$.
The valence of the vertex
$gG_{i+1}$ is one plus the cardinality
of the number of cosets of $G_i$ in $G_{i+1}$  (and 1 for level zero). Thus
this is a locally finite tree precisely when every $G_i$ has finite index
in $G_{i+1}$ and is a bounded valence tree exactly when $[G_{i+1}:G_i]$ is
bounded above over all $i$.

We therefore introduce the following finitary properties as variants
of property ($FA$):

\begin{defn} \label{faff}
  The countable group $G$ has the fixed point property ({\bf $FA_{lf}$})
  if every simplicial action  of $G$ on a locally finite tree has a global
  fixed point. It has the fixed point property ({\bf $FA_{bv}$})
  if every simplicial action of $G$ on a bounded valence tree has a
  global fixed point.
\end{defn}

In \cite{ser} I.6.3 it is shown that the property $(FA)$ is preserved by
taking quotients, extensions of groups with $(FA)$, and by taking
supergroups of finite index (but not subgroups of finite index). Exactly
the same proofs apply to $(FA_{lf})$ and $(FA_{bv})$. However we now
give examples to distinguish these three properties.
\begin{ex} Here is a general way to create finitely presented groups
which do not have $(FA)$ but which do have $(FA_{lf})$. Start with a
finitely presented seed group $G$ which itself has $(FA)$. 
Suppose that $G$ has a finitely generated subgroup $H$ of
infinite index in $G$ with the following
property: if $H$ is contained in a finite index subgroup $L$ of $G$ then
$L=G$. (This is sometimes known in the literature as $H$ is not engulfed
in $G$ or $H$ is dense in the profinite topology of $G$.) For instance
we have in \cite{cap} by Caprace and R\'emy finitely presented simple groups
$S$ with property (T), thus with $(FA)$, so we can here take $G=S$ and
$H=\{id\}$.

We then form the amalgamation $A=F*_HG$ where $F$ is any finitely
presented group with $(FA)$ that contains a copy of $H$ as a proper subgroup.
Now $A$ is finitely presented and has an unbounded action
on the Bass - Serre tree given by this splitting. But suppose that $A$ acts
on a locally finite tree. Then
$F$ fixes a vertex $v$ of this tree (subdividing if necessary) and $G$
fixes some vertex $w$. On taking a path from $w$ to $v$, the finite
valence means that the subgroup $L$ of $G$ which fixes
the path between $w$ and $v$ has finite index in $G$.
Now $H$ is a subgroup of both $F$ and $G$ and
so fixes both $v$ and $w$, thus $H$ is contained in $L$. But our chosen
property of $H$ means that $L$ is equal to $G$. Thus $G$ fixes the
path between $v$ and $w$, so certainly fixes $v$. But so does $F$,
meaning that $A=\langle F,G\rangle$ does as well.
\end{ex}

\begin{ex} \label{nlfbv}
  Let us now show that $(FA_{lf})$ is not equivalent to $(FA_{bv})$ over
  all countable groups. Consider the restricted direct product
$G=C_2\times C_3\times C_5\times C_7\times \ldots$. We can regard $G$ as a
direct union of the increasing sequence of subgroups
$C_2,C_2,\times C_3,C_2\times C_3\times C_5,\ldots $, so that the coset
construction above gives us an action of $G$ on a locally
finite tree $T_0$ but this tree has unbounded valence. Note that this action
of $G$ does not have bounded orbits: indeed all stabilisers are finite.

But now suppose $T$ is any tree with
valence bounded by $N$ and consider an element $g$ acting on $T$ with
$g$ having finite order coprime to the numbers 
$2,3,4,\ldots ,N$. Then $g$ must fix a vertex $v_0$ as
$T$ is a tree but it must then also fix the vertices adjacent to $v_0$
by Orbit - Stabiliser and so on, thus $g$ acts as the identity. This means
that if $p^+$ is the smallest prime which is greater than $N$
and $p^-$ the largest prime at most $N$ then any element
in the subgroup $C_{p^+}\times\ldots$ acts trivially on $T$. Thus the
quotient of $G$ by this subgroup which is in the kernel of the action
leaves us only the finite group $H=C_2\times C_3\times\ldots \times C_{p^-}$
and so any action of $G$ on $T$ will have bounded orbits and hence a
global fixed point.
\end{ex}

Note: Let $\Gamma$ be a graph with valence bounded above by
$b$ say. Suppose $\Gamma$ has the
property that there is an integer $c$ such that for
any finite group $F$ acting by automorphisms on $\Gamma$, there is some vertex
$v_F$
of $\Gamma$ whose orbit has diameter at most $c$ (so $c=0$ for a tree).
The referee points out that the argument above applies to any action
of $G$ on $\Gamma$. This is because the orbit of $v_F$, or equivalently
the index of the stabilizer of $v_F$ in $F$, must have size
at most $1+b+\ldots +b^c$. Thus once we have primes $p$ greater than this
then putting $F=C_p$ means that $C_p$ fixes the vertex $v_F$ and so by
the above, any action of $G$ on $\Gamma$ will have bounded orbits.
Examples of such graphs would be the (1-skeleta of) finite dimensional CAT(0)
cube complexes (as some cube must be stabilized under $F$), Helly graphs
(a fixed vertex in the first Helly subdivision of $\Gamma$) and (bridged
graphs, namely the 1-skeleta of) systolic complexes. 

However $G$ in Example \ref{nlfbv}
is clearly not finitely generated. In fact we have:
\begin{prop} \label{equprp}
  The properties $(FA_{lf})$ and $(FA_{bv})$ are equivalent over the class
  of all finitely generated groups.
\end{prop}
\begin{proof}
Suppose that our finitely generated group $G$ acts on the locally finite
tree $T$ without a global fixed point. Thus this is not an action with
bounded orbits and the fact that $G$ is finitely generated
means that it is not a parabolic action either. Therefore there is a
loxodromic element of this action and hence we can take the core
$C_G(T)$
which we know
is an infinite subtree of $T$ such that $G\backslash C_G(T)$ is a finite
graph. Now a finite graph clearly has bounded
valence and the vertices $v$ in $C_G(T)$ fall into
finitely many orbits under $G$, with the degree constant over each orbit
(though this could be much bigger than the degree of the 
corresponding vertex in the quotient graph), thus we have a
finite upper bound for the degree of vertices in $C_G(T)$.
\end{proof}

\subsection{Proper actions on products of trees}

We mentioned in the introduction that if a finitely generated group
$G$ acts on a tree such that all stabilisers are finite
then $G$ is a virtually free group.
However the purpose of this paper is to
consider actions of groups on products of trees, not just on a single tree.
Suppose that $P$ is a finite product $T_1\times\ldots \times T_n$ of
trees where the trees need not be isomorphic. We put a
cube complex structure on $P$ in the obvious way and we equip $P$ with
the $\ell_2$ product metric, using the path metric on each factor, so
that $P$ is a finite dimensional CAT(0) cube complex. If further all
trees are locally finite then $P$ is a locally finite cube complex and
also a proper metric space. We then consider groups $G$ acting
on $P$, so that there is a homomorphism from $G$ to $Aut(P)$ where
the cube complex structure is preserved (and so $G$ will be acting by
isometries of $P$).
This means that $G$ could act
by permuting the factor trees (although it can obviously only permute
isomorphic trees), though there will always be a finite index subgroup
$H$ of $G$, written $H\leq_f G$, which preserves all factors. Then
we can consider $H$, or at least the image of $H$ in
$Aut(P)$, as a subgroup of
$Aut(T_1)\times\ldots \times Aut(T_n)$, in which case we say that
$H$ acts {\bf preserving factors}. In this situation, we obtain for each
$1\leq i\leq n$ an action of $H$ on the tree $T_i$, given by the natural
projection $\pi_i:Aut(T_1)\times\ldots \times Aut(T_n)\rightarrow Aut(T_i)$.

\begin{defn} \label{prpd}
We say a group $G$ acts {\bf properly} on a finite product $P$
of trees if the stabiliser of every vertex in $P$ is a finite group.
\end{defn}
This property of groups is clearly preserved under subgroups, supergroups
of finite index (by inducing an action which will also be proper) and finite
direct products of groups. However it is not preserved by extensions in
general, for instance $\Z$ and $\Z^2$ clearly act properly on a bounded
valence tree and a product of two bounded valence trees respectively,
but the Heisenberg group $H(\Z)$ with $H(\Z)/\Z\cong\Z^2$ cannot act
properly on any finite product of (even locally infinite) trees as the
infinite order central element must fix a point when acting on any tree,
and this is true of any finite index subgroup as well.

Note that there are other notions of a proper action in the literature, such
as being metrically proper or
having a finite kernel and with the image being discrete under the compact
open topology/topology of pointwise convergence for $Aut(P)$. These notions
need not be the same as our Definition \ref{prpd} even for one tree.
However we note that, for instance, the definition here coincides with
that in \cite{gntr} which investigates proper actions on the more general
class of finite dimensional CAT(0) cube complexes. Moreover our main focus
in this paper is when our trees are all locally finite, in which case
all common definitions of a proper action coincide.
 
If our group $G$ acts preserving factors so that any $g\in G$ can be written
$(g_1,\ldots ,g_n)$ for some $g_i\in Aut(T_i)$ then 
$g$ is in the stabiliser $G_{\boldsymbol v}$ for a vertex
${\boldsymbol v}=(v_1,\ldots ,v_n)$
of $P$ if and only if $g_1$ stabilises 
$v_1\in T_1,\ldots ,$ and $g_n$ stabilises 
$v_n\in T_n$. In other words, given a vertex 
$v_i\in T_i$, suppose we write $G_{v_i}$ for the stabiliser of $v_i$ in $G$
when we regard $G$ as acting on the tree $T_i$ by projection. Then $G$ acts 
properly if and only if for all vertices $(v_1,\ldots ,v_n)\in P$, we have
the intersection of stabilisers $G_{v_1}\cap \ldots \cap G_{v_n}$ is finite.
In fact if all trees are locally finite then it is enough just to check
one vertex of the product space, because all such vertex stabilisers
will be commensurable with each other. However if some tree is not locally
finite then every vertex in the product needs to be checked for a
finite stabilizer.

We can now ask: when does a group act properly on a finite
product of trees? What about if we require each tree to be
locally finite or of bounded valence? The first point is that a group $G$
having property ($FA$) (or ($FA_{lf}$) or ($FA_{bv}$)) would act on any
(locally finite/bounded valence) tree with a global fixed point, so
if $G$ acts on any finite product $P$ of (locally finite/bounded valence)
trees preserving factors 
then it also has a global fixed point in $P$ and hence fails to
act properly if it is infinite. This might not be the case
if a finite index subgroup $H$ of $G$ fails to have this property, as $H$
could have a suitable action preserving factors whereupon we can induce
the action up to $G$. However if $G$ is infinite and
all finite index subgroups of $G$ have
($FA$) (or ($FA_{lf}$) or ($FA_{bv}$)) then $G$ cannot act properly on any
finite product of (locally finite/bounded valence) trees. Moreover our group
$G$ cannot contain any such subgroup $S$ either, as otherwise the proper
action would pass from $G$ down to $S$. 

We now show, as was seen in Proposition \ref{equprp},
that if we are dealing with finitely generated
groups then we do not need to worry about the distinction
between locally finite and bounded valence trees.

\begin{lem} \label{equiv}
If $G$ is a finitely generated group then $G$ acts
properly on a finite product of locally finite trees if and only
if $G$ acts properly on a finite product of bounded valence trees.
\end{lem}
\begin{proof}
If the trees $T_1,\ldots ,T_k$ are locally finite,
we can suppose that $G$ preserves factors by dropping down
to a finite index subgroup and then taking an induced action at the
end. Let us therefore consider the action of $G$ on each 
$T_i$. As $G$ is finitely
generated, the core $C_G(T_i)$ has bounded valence by Proposition
\ref{equprp}, so we replace each $T_i$ in the product with its core
$C_G(T_i)$, which is invariant under $G$ and with the action on
$C(T_1)\times\ldots\times C(T_k)$ still proper since stabilisers
of vertices in the product are still finite.
\end{proof}

As for a group $G$ which shows that this is not true without the
finite generation condition on $G$, we can adapt Example \ref{nlfbv}.
\begin{ex} \label{nlfbvag}
  Our group $G=C_2\times C_3\times C_5\times \ldots$ was shown to act
  on a locally finite tree with finite stabilisers, but on any bounded
  valence tree it must have a global fixed point. This argument also
  applies to any group of the form $C_{p_1}\times C_{p_2}\times\ldots$
  for $p_1,p_2,\ldots$ a strictly increasing sequence of primes. 
  If we now take $H$ to be any index $N$ subgroup of $G$ then for any
  factor $C_p$ of $G$ with $p$ prime and $p>N$, we have $H\cap C_p$ has
  index at most $N$ in $C_p$ so $C_p<H$. Thus $H$ contains the finite index
  subgroup $C_{p^+}\times \ldots$ for $p^+$ the smallest
  prime larger than $N$, which has ($FA_{bv}$) and hence so does $H$.
  Thus $G$ cannot act properly on a finite product of bounded valence trees,
  even without preserving factors.
\end{ex}

One useful consequence of the above is that if we have a finitely generated
group $\Gamma$ which contains $G$ in Example \ref{nlfbvag} then $\Gamma$
also cannot act properly on a finite product of bounded valence trees, so
by Lemma \ref{equiv} $\Gamma$ cannot act properly on a finite product of
locally finite trees.

\subsection{Groups possessing proper actions}

Let us now consider which particular groups are known to act properly on a
finite product of (bounded valence/locally finite) trees.

Starting with the case of one tree, it is a well known fact
that a group $G$ acts freely (namely with trivial stabilisers)
on some tree $T$ if and only if it is a free group. (For instance see
\cite{ser} Proposition 15 and Theorem 4 for a combinatorial proof,
whereas a topological proof follows immediately from the statements
in \cite{htch} Proposition 1.40, Proposition 1A.2 and Lemma 1A.3.)
However if we now change
acting freely to acting properly then we can have infinitely generated
groups acting with finite stabilisers on locally finite or even bounded
valence trees but which are not virtually free. But if we restrict to
finitely generated groups $G$ then
(see for instance \cite{ser}, \cite{dunw}) $G$ acts properly on some tree
if and only if $G$ is virtually free, in which case we can take this
tree to be of bounded valence.

If we now consider the case of $k\geq 2$ trees, we do have other examples
(whereupon we can take subgroups, finite index supergroups and direct
products to create further examples). The most obvious groups are
direct products of $k$ free groups where the $i$th factor group acts 
freely on the $i$th factor tree and trivially on all other factor 
trees. Other examples for $k=2$ include the Burger - Mozes
groups in \cite{bm} and the related groups of Wise in \cite{wshlv} using
complete square complexes (see \cite{cpbmw} for a useful survey about
these groups, especially Section 4). These examples in fact act properly and
cocompactly. We also have lamplighter groups of the form $C_n\wr \Z$
for any integer $n\geq 2$. These groups are finitely generated but are well
known not to be finitely presented. Thus they
cannot act properly and cocompactly on any product of trees,
for instance by \cite{bh} I.8.11 which states that a group is finitely
presented if and only if it acts properly and compactly on a simply
connected geodesic space.

To see that these lamplighter groups $C_n\wr \Z$ do act properly on a finite
product of locally finite trees, one approach (when $n=p$ is prime)
is to note that they embed in $PGL(2,\F_p(t))$ for $t$ a transcendental
element over the finite field $\F_p$ via the matrices
\[\sma{ll} 1&1\\0&1\fma \mbox{ and }\sma{ll} t& 0\\0 &1\fma.\]
One can then take a couple of valuations on $\F_p(t)$ and
use the action of $PGL(2,\F_p(t))$ with respect to each valuation
on its Bruhat - Tits building (here a tree).

As an alternative, suggested by the referee and which can be adapted to
other examples where $C_n$ need not be prime, or even can be replaced
by any finite non-trivial group,
we vary the coset construction at the start of Section 2 by
allowing the case where $G$ is an ascending union $\cup_{i=-\infty}^\infty G_i$
of subgroups, with $i$ now varying over $\Z$ rather than $\N$. This action
has exactly the same properties as the previous coset construction, except
that now the level 0 vertices are not distinguished and the quotient
graph $G\backslash T$ is a bi-infinite line. Also $T$ has infinitely many
ends but there is still a distinguished end that is fixed by every
element of $G$, which corresponds to taking a ray where the vertex
stabilisers increase as one moves further out along the ray. Note that if
each subgroup $G_i$ is properly contained in $G_{i+1}$ then every $G_i$ must
be infinite and so stabilisers of vertices will be infinite. This action is
still a parabolic/horocyclic action with every element acting elliptically.

However now suppose we have a group $\Gamma=\langle H,t\rangle$ which is
a strictly ascending HNN extension of some group $H$. That is, $H$
(which may or may not be finitely generated) contains a proper subgroup
$L$ which is isomorphic to $H$ via some isomorphism $\theta:H\rightarrow L$
and  $t$ is the stable letter conjugating
$H$ to $tHt^{-1}=L$, where $tht^{-1}=\theta(h)$ for $h\in H$. This can
also be described as a semi-direct product
$\Gamma=G\rtimes\langle t\rangle$ where
for $i\in \Z$ we set $G_i=t^{-i}Ht^i$ and
$G=\cup_{i=-\infty}^\infty G_i$ is a strictly ascending union with
$tG_it^{-1}=G_{i-1}$, because $L$ is a proper subgroup of $H$.

Now consider the Bass - Serre tree $T$ of this HNN extension. On restricting
the corresponding action of $\Gamma$ to $G$, we obtain the variation
of the coset construction over $\Z$ rather than $\N$ which we have just
described. As for $t$, it sends the vertex $gG_i$ to the vertex
$tgG_it^{-1}=tgt^{-1}G_{i-1}$, thus acts as a loxodromic element with axis
consisting of the vertices defined by the subgroups $G_i$. Note that
although the action of $G$ on $T$ is of
parabolic/horocyclic type, the action of $\Gamma$ is of
quasi-parabolic/focal type and the distinguished fixed end is given by
the repelling fixed point of this stable letter $t$.

We can also approach this construction the other way round: suppose we
have a group $\Gamma=\langle H,t\rangle$ where $H$ is a subgroup of $\Gamma$
and $t$ is an element of $\Gamma$ such that $tHt^{-1}$ is a proper
subgroup of $H$. Then $\Gamma$ is a strictly ascending HNN extension of
$H$. This can be seen because the natural homomorphism from the abstract
HNN extension onto $\Gamma$ is injective, since every element in the
HNN extension can be written as $t^iht^j$ for $h\in H$ and $i,j\in \Z$
but any such element is non-trivial in $\Gamma$ unless $h=id$ and $i+j=0$.

Now in order to make $C_2\wr\Z$ (or $F\wr\Z$ for $F$ any finite group, on
merely replacing $C_2$ with $F$) act in this way, we write this group
as
\[\left(\bigoplus_{j=-\infty}^\infty F_j\right)\rtimes\langle t\rangle\]
where $F_j=\{id, x_j\}$ is a copy of $C_2$ and $tx_jt^{-1}=x_{j+1}$.
We then set $G_i$ to be $\oplus_{j=-\infty}^i F_j$, thus expressing
$G=\oplus_{j=-\infty}^\infty F_j$ as an ascending union of the subgroups $G_i$.
As we have $tG_it^{-1}=G_{i+1}$ with $G_i$ properly contained in $G_{i+1}$,
we see that $\Gamma$ is a strictly ascending HNN extension
(though here the stable letter is $t^{-1}$ if we want it to be
consistent with the description above).
As $G_i$ clearly has index 2 (or index $|F|$) in $G_{i+1}$, the resulting
Bass - Serre tree of this HNN extension has bounded valence. However let us
consider the stabiliser in $\Gamma$
of the vertex corresponding to $G_i$. Using the
semi-direct product decomposition, any element in $C_2\wr\Z$ can be
expressed in the form $gt^k$ for $g\in G$ and $k\in\Z$. But $gt^k$ sends
the $i$th level vertices to the $k+i$th level, so if $k\neq 0$ then
$gt^k$ does not fix a point (and so is loxodromic). Thus the stabiliser
in $\Gamma$ of this vertex is actually the stabiliser in $G$, which
is equal to $G_i$. In particular
this action is not proper (indeed there cannot be a proper action
of $G$ on any tree since $G$ is finitely generated but clearly not
virtually free).

However we can make $C_2\wr\Z$ (or $F\wr\Z$) act properly on the product
of two locally finite trees by now running this process ``backwards''.
Namely we create a second action of $C_2\wr\Z$ on this locally finite tree
by using the same form of coset construction, but this time defining
$G_i$ to be the subgroup $\oplus_{j=i}^\infty F_j$ for $F_j$ the same copy
of $C_2$ as before. On taking $v$ to be the vertex of the first tree
corresponding to (and thus having stabiliser equal to)
$G_0=\oplus_{j=-\infty}^0 F_j$, and $w$
to be the vertex of the second tree corresponding to the subgroup
$\oplus_{j=0}^\infty F_j$, we have that when $\Gamma$ acts on the product
of these two trees (preserving factors), the stabiliser in $\Gamma$ of the
vertex $(v,w)$ in this product is equal to the stabiliser in $G$
of $(v,w)$, which is
\[\left(\bigoplus_{j=-\infty}^0 F_j\right)\cap
    \left(\bigoplus_{k=0}^\infty F_k\right)=F_0\]
  which is finite, thus $C_2\wr\Z$ acts properly on the product of
  two locally finite trees.

  We can use this construction to give an example of a finitely generated
  group that acts properly on a finite product of trees, but not on
  any finite product of locally finite trees. The author is grateful to
  the anonymous referee for pointing out this result.
  \begin{thm} \label{ref}
    Consider the second Houghton group $H_2=FSym(\Z)\rtimes\langle t\rangle$
    where $FSym(\Z)$ is the (infinitely generated) group of permutations    
    of the integers $\{\ldots ,-1,0,1,\ldots\}$ with finite support and
    $t$ is the shift $t(k)=k-1$. Then $H_2$ is finitely generated and acts
    properly preserving factors
    on the product of two trees, but does not act properly (whether
    preserving factors or not) on the product of any finite number of
    locally finite trees.
  \end{thm}
  \begin{proof}
    The group $H_2$ is well known to be finitely generated, for instance
    by the transposition $(0\,\,1)$ and $t$. For $i\in\Z$, let us define
    the subset $S_i$ of $\Z$ to be all integers at most $i$. This gives 
rise to the ascending
sequence $G_i=FSym(S_i)$ of subgroups of $G=FSym(\Z)$, where $FSym(S_i)$
is the group of permutations having support that is finite and contained in
the set $S_i$. Note that $tG_it^{-1}$ is equal to $G_{i-1}$ and $G_{i-1}<G_i$
so this is a strictly ascending HNN extension as above. We then run the
process backwards by considering $FSym(S'_i)$ where $S'_i$ consists of all
the integers at least $i$, and then take the resulting action of $H_2$ on
the product of these two trees.
Note though that these are not locally finite trees because $FSym(S_{i-1})$
has infinite index in $FSym(S_i)$ by Orbit - Stabiliser, since
the orbit under $FSym(S_i)$ of $i$ is infinite.

We now look for the stabiliser of a vertex in the product structure,
but as the trees are not locally finite we will need to check the
stabiliser of more than one vertex in each tree. However as our actions
come from the coset construction, it is enough to check the stabiliser
of each vertex corresponding to a subgroup $G_i$ for all large $i$,
because every vertex stabiliser lies in one of these.

In the first tree the stabiliser of the vertex labelled by $FSym(S_i)$ is
just the subgroup $FSym(S_i)$ and similarly we obtain the stabiliser
$FSym(S'_j)$ for the appropriate vertex in the second tree. But any group
element lying in both these stabilisers fixes all points greater than
$i$ and all points less than $j$, so
the intersection of these two stabilisers is finite.

Now note that by taking disjoint cycles in $\Z$ of
length $p$ for each prime $p$, we have a copy in $H_2$
of the group $G$ from Example \ref{nlfbvag} where it was shown that
$G$ and all its finite index subgroups have property ($FA_{bv}$).
Thus $H_2$ cannot act properly on any finite product of bounded valence,
hence locally finite by Lemma \ref{equiv}, trees.
\end{proof}

This also suggests (as indicated by the referee) that we can find
examples with better finiteness properties by considering the higher
Houghton groups $H_n$ for $n\geq 3$. It is indeed the case that they
also satisfy the conditions of Theorem \ref{ref} (with two replaced
by $n$). Here we give an outline of that argument for $H_3$, which
is well known to be finitely presented.

\begin{co} \label{h3fp}
    The finitely presented third Houghton group $H_3$ acts
    properly preserving factors
    on the product of three trees, but does not act properly (whether
    preserving factors or not) on the product of any finite number of
    locally finite trees.
\end{co}
\begin{proof}
    The argument for locally finite trees is exactly the same as in Theorem
    \ref{ref} because $H_3$ contains the same poison subgroup preventing
    an action on bounded valence trees and $H_3$ is of course finitely
    generated.

    For the action on three trees, the facts we will use about $H_3$
    (see \cite{cxy} and references therein) are:\\
\hfill\\
(1) $H_3$ is the subgroup of bijections of the (``three pronged'') set
    $X=\{(i,m)\,|\,i=1,2\mbox{ or }3\mbox{ and }m\in\Z\}$ consisting of those
    bijections which are ``eventually a translation'' on each of the three
    prongs, where the first prong $P_1$ is equal to the
    subset $\{(1,m)\,|\,m\in\Z\}$ of $X$, and the same for $P_2,P_3$.\\
\hfill\\
(2) There is a homomorphism $\tau:H_3\rightarrow \Z^3$ given
    by $\tau(g)=(t_1(g),t_2(g),t_3(g))$ where $t_i(g)$ is the
    ``eventual translation length'' on prong $P_i$
    (but it is not onto as we have $t_1(g)+t_2(g)+t_3(g)=0$).\\
\hfill\\    
(3) We have three distinguished infinite order elements $h_1,h_2,h_3$
    where (taking $i$ modulo 3)
    $h_i$ pushes points up the $i$th prong, namely $h_i(i,m)=(i,m+1)$,
    down the $i+1$th prong and fixes $P_{i+2}$ pointwise. Thus
    $\tau(h_1)=(1,-1,0)$ and so on.\\
\hfill\\
    To obtain the action of $H_3$ on the first tree $T_1$, let $K$
    be the kernel of the homomorphism from $H_3$ to $\Z$ given by
    projecting $\tau$ onto the first $\Z$ factor. Then we have
    $H_3=K\rtimes \langle h_1\rangle$ and $K$ consists of elements
    that are eventually the identity on $P_1$.

    Thus for $j\in\N$ define subsets $S_j$ of $X$ by taking the union of
    $P_2,P_3$ and the first $j$ points of $P_1$. Then $S_j$ is an increasing
    sequence of sets with union $X$ and $h_1(S_j)=S_{j+1}$. Hence we can
    write $K$ as the increasing union of subgroups $K_j$ for $j\geq 0$
    where $K_j$
    fixes pointwise the complement of the set $S_j$ (namely the points
    $j+1,j+2,\ldots$ on the first prong). But we have $h_1K_jh_1^{-1}=K_{j+1}$
    for $j\geq 0$ (and we can use this expression to define $K_j$ for $j<0$
    too) with $K_j$ strictly contained in $K_{j+1}$, so we can create our tree
    $T_1$ from this strictly ascending HNN extension.

    We can do the same on the other two prongs, expressing the kernels
    $L$ and $M$ as $\cup_{j=-\infty}^\infty L_j$ and $\cup_{j=-\infty}^\infty M_j$
    respectively. Thus we have created an action
    of $H_3$ on three trees and hence on the corresponding
    product space $C$. We now look for the stabilisers of each product vertex
    in $C$, where it is enough just to check those of the form
    $(K_{j_1},L_{j_2},M_{j_3})$ coming from the respective
trees. We can assume that $j_1,j_2,j_3$ are all positive because that
only increases the stabiliser. But for an element $g\in H_3$ to be in the
subgroup $K_{j_1}$ coming from the action on $T_1$, we must have $g$ fixing
all points on the first prong above $j_1$. The same occurs on the other two
prongs with $j_2$ and $j_3$, so $g$ can move at most $j_1+j_2+j_3$ points
of $X$ and thus this stabiliser is finite.
\end{proof}

\section{Obstructions to acting properly on a product of trees}

Having seen in the last section some finitely generated groups which
do act properly on a finite product of locally finite trees, we might
now ask for obstructions that prevent this from happening. We have
already noted the obstruction of $G$ containing an infinite subgroup $H$
where $H$ and all its finite index subgroups have property ($FA$) or
($FA_{lf}$) or even ($FA_{bv}$). More obstructions can be found using
the fact that if our group $G$ had such an action then this would
be an example of a proper and semisimple action on a CAT(0) space.
Thus $G$ will need to possess the consequences of this given in
\cite{bh} III.$\Gamma$.1.1, although this would also apply if $G$ merely
had a proper action on a finite product of trees, not necessary
locally finite. As for the relationship between these obstructions,
we certainly have finitely generated groups which do not have property ($FA$)
but which cannot act properly and semisimply on a CAT(0) cube complex.
However, going the other way,
we note the recent and very interesting example in \cite{kjw} of
a compact non positively curved cube complex where the fundamental group
and all of its finite index subgroups (there is only one) have property (FA).

Another obstruction is failing to have finite asymptotic dimension.
We do not define this here since
the paper \cite{bd} gives an introduction to asymptotic dimension, as
well as a thorough survey of the asymptotic dimension of various groups.
However (as pointed out by the referee) the property of a group $G$
acting properly on a finite product of locally finite trees can be
regarded as the equivariant version of having finite asymptotic dimension.
This can be seen from \cite{bd} Theorem 39 which states that any proper
metric space (such as a finitely generated group with the word metric
from a finite generating set)
with asymptotic dimension at most $n$ admits a coarse
embedding into the product of $n+1$ locally finite trees. (The $+1$ is needed
here because our lamplighter groups have asymptotic dimension 1, but suppose we
have a coarse embedding $f$ of a finitely generated group $G$ in some tree $T$.
Then $f(G)$ will be coarse connected but any coarse connected subspace of $T$
is itself quasi-isometric to some tree $T'$. Thus the coarse equivalence
from $G$ to $T'$ is in fact a quasi-isometry and so $G$ would be
virtually free.) On the other hand, if a finitely
generated group $G$ acts (metrically) properly on a space $X$ which is
a product of $n$ locally finite (unbounded) trees then $X$ has asymptotic
dimension $n$ and we can pull back the metric on $X$ to one on
$G$ using this action. (Strictly speaking this may only be a pseudo-metric
but only finitely many points will have distance 0 from any given point
and these results still apply here.) Now our new metric on $G$ might not
be quasi-isometric to the word metric but it is coarse equivalent, thus
giving us our coarse embedding of $G$ into a space of asymptotic dimension
$n$ and so $G$ has asymptotic dimension at most $n$.  

We have already seen, consistent with these results, that $C_2\wr\Z$ and
indeed $F\wr\Z$ (where $F$ is any finite non-trivial group) acts properly
on a product of two trees. Any group acting properly on a product of $n$
trees clearly acts properly on an $n$ dimensional CAT(0)
cube complex. The question
of when $F\wr H$, for $H$ a selection of various groups, acts properly on a
finite dimensional cube complex has been considered, particularly by
Genevois in \cite{gnphd} and \cite{gntr}. In \cite{gntr}
Proposition 1.3 it is shown that $F\wr F_2$ (and hence $F\wr H$ for
any group $H$ containing a non abelian free group) cannot act properly
on any finite dimensional CAT(0) cube complex and so certainly not on any
finite product of trees. This suggests (as pointed out by the referee)
considering much smaller groups $H$. In \cite{gnphd} Proposition 9.33
it was shown that for every $n\geq 1$ the group $F\wr\Z^n$ acts
(metrically) properly on a $2n$ dimensional CAT(0) cube complex.
Therefore the obvious case for us to try here is $H=\Z^2$, whereupon
we obtain the same dichotomy as we saw for the Houghton groups.

\begin{thm}
  For $F$ any finite non-trivial group,
  the group $G=F\wr \Z^2$ acts properly preserving factors on a product of
  four trees but does not act properly on any finite product of locally finite
  trees.
\end{thm}
\begin{proof}
  We will assume throughout that $F=C_2$, as in general this just changes the
  notation but does not affect the actual argument.
  We write $G=K\rtimes\langle s,t\rangle$ where
  \[K=\langle x_{i,j}\,|\,i,j\in\Z\rangle,\]
so that each $x_{i,j}$ has order 2 and $sx_{i,j}s^{-1}=x_{i+1,j}$,
$tx_{i,j}t^{-1}=x_{i,j+1}$. Also any element $g\in G$ can be written
uniquely as
\[g=\left( \prod_{(i,j)\in S_g} x_{i,j}\right) s^{m_g}t^{n_g}\]
for $m_g,n_g\in\Z$ and $S_g$ some finite subset of $\Z^2$.

We suppose that $G$ acts properly on a finite product of locally finite trees.
Then by Lemma \ref{equiv} we can assume that $G$ acts properly on
$T_1\times\ldots\times T_D$ where each tree $T_r$ has bounded valence.
We can even assume that $G$ acts properly on such a product
whilst preserving factors because we can drop down to a finite index
subgroup $H$ of $G$ which will do this, but any such $H$ will contain a
copy of $G$.

As $G$ does not contain a free group but is finitely generated, any tree
action of $G$ (whether on a locally finite tree or not) must be lineal
or quasi-parabolic/focal (we can throw away from the product $P$ any action
with a global fixed point as this will not affect whether the
action on $P$ is proper). We further assume for now that each action 
is quasi-parabolic/focal.

Let us consider such an action of $G$ on any tree $T$ which has valence
bounded above by $B$ say. There must be some primitive element $p=s^at^b$
of $\Z^2$ that is acting loxodromically on $T$, as otherwise all elements
of $G$ would act elliptically. But there will also be another element
$q=s^ct^d\in\Z^2$ which acts elliptically and such that $p,q$ generate $\Z^2$.
To see that we can do this, note
that there is much freedom here for the element $p$ but $q$ is uniquely
defined up to its inverse. In order to proceed with definite choices,
we first take a suitable $c$ and $d$ which gives us a primitive
elliptic element $q$, then take any $a$ and $b$ such that $p$ and $q$
generate $\Z^2$. Now $p$ will be loxodromic (as otherwise $\Z^2$
will act elliptically) and we will have $ad-bc=\pm 1$. We can now insist
that $ad-bc=1$ by changing $q$ to its inverse if needed.

This means that for any element $g\in G$ as written in the form above, we can
replace the $s^{m_g}t^{n_g}$ factor in this expression by the same element
but written as $p^iq^j$. We now have that $g$ is elliptic in our action
if and only if $i=0$. Replacing the loxodromic element $p$ with its inverse
if necessary which changes nothing else above (as long as we remember to
invert $q$ too so that $ad-bc$ is still 1), we take the axis $A$
of $p$ and consecutively label the vertices $\{v_i\,|\,i\in\Z\}$
in such a way that
the attracting fixed point of $p$ is the fixed end $E:=(v_0,v_1,v_2,\ldots)$
of this action. Then for any $i\in \Z$ we will have $p(v_i)=v_{i+\tau}$, where
$\tau\in\N$ is the translation length of $p$.
Note that as $q$ commutes with $p$ and is elliptic, it fixes the axis
$A$ pointwise.

This still allows a choice of base vertex $v_0$ by shifting the subscript,
so we proceed as follows: the element $x_{0,0}$ is elliptic and therefore
fixes some vertex of $T$, but every element of $G$ fixes the end $E$.
Therefore $x_{0,0}$ fixes the vertices $v_i$ for large $i$ and any
group element fixing a particular $v_{i_0}$ also fixes $v_i$ for all
$i\geq i_0$. We therefore choose $v_0$ such that $x_{0,0}$ fixes $v_0$ but not
$v_{-1}$ (if $x_{0,0}$ fixes all of $A$ then so will all $x_{i,j}$, whereupon
this action of $G$ would be lineal).
We are now interested in the subgroup $V_0:=Stab(v_0)$ which will lie in
$\langle K,q\rangle$,
so we first consider which other $x_{i,j}$ do or do not fix $v_0$.

On taking the elements
$\gamma_{i,j}:=p^iq^jx_{0,0}q^{-j}p^{-i}$ over all $i,j\in\Z$, we see that these
are just the elements $\{x_{i,j}\,|\,i,j\in\Z\}$ written differently because
$\langle p,q\rangle=\langle s,t\rangle$. Now a given $\gamma_{i,j}$
certainly fixes $p^i(v_0)=v_{i\tau}$. If moreover $i\leq 0$ then
$\gamma_{i,j}$ will also fix $v_0$ and hence lies
in $V_0$. If however $i>0$ then by reversing this argument, claiming that
$\gamma_{i,j}$ also fixes $v_0$ would imply that
$x_{0,0}$ fixes $v_{-i\tau}$ and hence also $v_{-1}$ which is a contradiction.
Hence for any $I\in\Z$
let us now define the subset $S_I$ of $K$ as $\{\gamma_{i,j}\,|\,i\leq I\}$
and the subgroup $K_I\leq K$ as $\langle S_I\rangle$. We have just argued
that $V_0\cap K_I$ is all of $K_I$ for $I\leq 0$ but not otherwise. Now we
bring in the bounded valence of the tree $T$: the action of the subgroup
$K_1$ is such that the vertex $v_\tau$ is fixed by all of $K_1$, so $v_0$
can only go to a descendant of distance $\tau$ from $v_\tau$ of which there
are at most $B^\tau$ possibilities. So by Orbit - Stabiliser, the index
of the stabiliser $V_0\cap K_1$ in $K_1$ is at most $B^\tau$. We can continue
this argument to get that $V_0\cap K_I$ has index at most $B^{I\tau}$ in $K_I$.
Indeed we can say that if $S$ is any subset of the elements $\gamma_{i,j}$
such that we have $I\in\Z$ with $i\leq I$
for any $\gamma_{i,j}\in S$ then for the action of $\langle S\rangle$ on $T$,
the stabiliser $V_0\cap\langle S\rangle$ has index at most $B^{I\tau}$ in
$\langle S\rangle$ if $I>0$ and is all of $\langle S\rangle$ if $I\leq 0$.

We now need to convert this statement back from the elements $\gamma_{i,j}$
to the elements $x_{i,j}$ in order to be able to consider different tree actions
of $G$ simultaneously. We have that $\gamma_{i,j}=x_{ia+jc,ib+jd}$ so that
the subset $S_I$ defined above contains $x_{m,n}$ exactly when $dm-cn\leq I$.
In particular we have that $V_0\cap K_I$ is all of $K_I$ exactly
when $dm-cn\leq 0$, which can be thought of those lattice points
$(m,n)\in\Z^2$ which lie on or to one side of the line $\ell$ through $(0,0)$
and $(c,d)$. We can argue similarly for $dm-cn\leq I$ where we take the
appropriate line parallel to $\ell$. As for knowing which is the correct
side, notice that either $(d,-c)$ or $(-d,c)$ could be taken as a normal
vector to $\ell$, but $(d,-c)$ lies on the same side of $\ell$ as $(a,b)$
by taking the dot product (since we insisted $ad-bc=1$). Thus it is the
other half plane which gives the points $(m,n)$ such that $dm-cn<0$ and
we can argue similarly for $dm-cn\leq I$ where we take the
appropriate line parallel to $\ell$ along with the half plane on the other
side of our normal $(d,-c)$ which we have obtained from this action.

Now we examine all of these $d$ actions on our trees together. As described
above, the action on tree $T_r$ comes with a base vertex $v_0^{(r)}$ and
our normal as obtained above, which we write as the vector
$w_r=(d_r,-c_r)\in\Z^2$, where $\langle s^{c_r}t^{d_r}\rangle$ generates
precisely the elements in $\langle s,t\rangle$ which
act elliptically on the tree $T_r$. Also we assume that $w_r$ is oriented
towards the same half plane as is $(a_r,b_r)$, where 
the attracting fixed point of the loxodromic element
$s^{a_r}t^{b_r}$ is the fixed end of
the action of $G$ on $T_r$. From these $D$ vectors $(d_r,-c_r)$
(which may be repeated
but only finitely many times so it does not matter), we create a Newton
polygon $P$. This is a compact convex polygon in $\R^2$ such that all
vertices lie in $\Z^2$ and such that the collection of vectors given
by the outward pointing normal obtained from each edge of $P$ is exactly our
$D$ vectors, up to rescaling each vector by a positive number.

To see the existence of $P$, first rescale all directions so that they
are unit length and consider them as points on the unit circle $S^1$. On
taking tangents at each point, we obtain a polygon given by the intersection
of each of the insides of these tangents, where an inside is the closed
half space with boundary the tangent line and containing the origin. This will
result in a convex polygon with the origin as an interior point. If it is not
compact then add finitely more such actions of $G$ on trees (whereupon we
still have properness of the action on the product) with appropriate
directions until it is compact. Now note that rationality of this process
means that every vertex of our compact polygon lies in
$\Q^2\setminus\{(0,0)\}$, so we can multiply by a large positive
integer until each vertex lies in $\Z^2\setminus\{(0,0)\}$ and such that
the unit disc lies in this polygon, which will be our initial Newton polygon
$P$. Note that we can further rescale $P$ by any integer $R\in\N$,
resulting in the enlarged Newton polygon $P_R$ with all the above properties,
and now such that $P_R$ contains every lattice point $(m,n)$ of $\Z^2$ with
$m^2+n^2\leq R^2$.

Let us now consider the elements in $G$ which lie in the stabiliser of
the product vertex $(v_0^{(1)},\ldots , v_0^{(D)})$, i.e. the intersection
of each stabiliser $V_0^{(r)}$ of $v_0^{(r)}$ for $r=1,2,\ldots ,D$. Our
assumption that this action is proper means that
$V_0^{(1)}\cap\ldots\cap V_0^{(D)}$ has order $C\geq 1$ say. First note
that we can assume any such element is actually in $K$ since every other
element will be loxodromic in some tree action as soon as there are two
non-opposite normal vectors. This will occur because we made sure $P$ was
compact. Then on taking the Newton polygon $P_R$ for some large
$R\in\N$ as defined above, we will now consider the finite subgroup $F_R$
of $K$ generated by the $x_{m,n}$ over all lattice points $(m,n)$ with modulus
at most $R$. By construction, every such point $(m,n)$ lies inside or
on $P_R$ and hence lies in the half plane away from the outward facing normal
of each edge. In particular these $(m,n)$ satisfy
$d_rm-c_rn\leq RI_r$, where $I_r$ is the value obtained from this edge
of the original polygon $P$. Note that $I_r>0$ (because $(0,0)$ is
an interior point of $P$)
and is independent of $R$, given that we multiply $I_r$ by $R$ to obtain
the correct value in the rescaled polygon $P_R$, as is $c_r$ and $d_r$.
Now for a given $R$, we have for $1\leq r\leq D$
that the stabiliser in $F_R$ of the vertex
$v_0^{(r)}$ has index at most $B^{RI_r\tau_r}$ where $\tau_r$ and of course $B$
are also independent of $R$.

If we now apply this for each of our $D$ actions, this implies that the
subgroup of $F_R$ fixing $(v_0^{(1)},\ldots , v_0^{(D)})$ has index
at most $B^{DRI_0\tau_0}$, where $I_0$ is an upper bound for the $D$ values of
$I_r$ and $\tau_0$ is for $\tau_r$.
Thus this subgroup has order at least $|F_R|/B^{KR}$ where $K$ is
independent of $R$. But the number of integer lattice
points in the closed disc of radius $R$ is bounded below by at least $R^2$
and so $F_R$ has order at least $2^{R^2}$. Thus the stabiliser in $F_R$ of this
product vertex has order tending to infinity as $R$ tends to infinity.
Hence we have a contradiction as soon as this order is more than $C$.

If any of the original tree actions were lineal but we still tried to go
through the above process for such an action, we would find that every
element in $G$ fixes setwise the two limit points of this action. If
they are fixed pointwise then every elliptic element of $G$, which will
include all elements of $K$, will fix every vertex on the axis of
of the resulting
loxodromic element. Thus we can ignore this action when considering
stabilisers. If some of our lineal actions swap the two limit points,
we can drop to a finite index subgroup $H$ of $G$ where all lineal actions
fix both limit points pointwise, then use the trick above that $H$ will
contain a copy of $G$ and hence restrict the original action of $G$ to this
copy.

As for the proper action of $G$ on four trees, merely mimic the above
with the four horizontal and vertical normals: namely first write
$G$ as $L\rtimes\langle s\rangle$ and $L$ as $\cup_{k=-\infty}^\infty L_k$ where
$L_k=\langle t,x_{i,j}
\mbox{ for }i\leq k\mbox{ and }j\in\Z\rangle$. Then
create the
Bass - Serre tree action from this expression of $G$ as a strictly
ascending HNN extension so that $s$ acts loxodromically. The
stabiliser of a vertex $v_k$ say on the axis of $s$ will be the subgroup
$L_k$. On now running this backwards and then swapping $s$ and $t$ (and $i$ and
$j$), the stabiliser of the product vertex given by four points with a pair
on each axis will be the subgroup generated by various $x_{i,j}$ where there are
bounds both above and below for $i$ and $j$, thus this stabiliser is finite.
\end{proof}

\section{\bf Obstructions for hyperbolic groups}  

We have been using the fact that when a group $G$ acts on a product $P$
of $n$ trees preserving factors, it acts properly if and only if
for all vertices $(v_1,\ldots ,v_n)\in P$, we have
the intersection of stabilisers $G_{v_1}\cap \ldots \cap G_{v_n}$ is finite.
Thus when $G$ is a torsion free group, it
acts properly on a finite product $P$ of trees if and only if it acts
freely on the vertices of $P$ (and indeed if and only if $G$ acts freely
on $P$ because if some element $g\in G$ stabilises a cube then a power
of $G$ will stabilise its vertices).

Using this, we can give an alternative criterion for groups of this kind
to act properly on $P$ if they preserve factors,
which nevertheless is very useful
because it involves looking at the components of the
elements of $G$ rather than the vertices
of $P$.
\begin{prop} \label{hyel}
  Suppose that the group $G$ does not contain an infinite torsion subgroup
  and acts on a finite product $P=T_1\times \ldots \times T_n$ of
trees preserving factors. Then $G$ acts properly on $P$
if and only if for any infinite order element
$g=(g_1, \ldots , g_n)$ in $G$ with each $g_i$ regarded as an element
of $Aut(T_i)$ under the natural projection, we have that at least
one $g_i$ from $g_1,\ldots ,g_n$ acts as a loxodromic isometry on the
tree $T_i$.
\end{prop}
\begin{proof} If we have an infinite order element $g=(g_1,\ldots ,g_n)\in G$   
where each $g_i$ acts as an elliptic element on $T_i$ then there will
be a vertex $v_i\in T_i$ fixed by $g_i$ so that the point 
$(v_1,\ldots ,v_n)\in P$ is fixed by $g$ and all its powers, thus the
action is not proper. Conversely if $g_i$ acts as a loxodromic element
on $T_i$ then no vertex $v_i\in T_i$ is fixed by $g_i$, so no vertex
${\boldsymbol v}=(v_1,\ldots ,v_k)$ in $P$ is fixed by $g$. Thus the
stabiliser of ${\boldsymbol v}$ can only contain torsion elements.
\end{proof}  

If $G$ in Proposition \ref{hyel} acts on a finite product of trees without 
preserving factors then this result is still useful, because as mentioned
we can take a finite index subgroup $H$ of $G$ which does act preserving
factors and then see if this condition applies to $H$.

Having considered in the previous sections groups which contained infinite
torsion subgroups, we now concentrate on those that do not, such as
hyperbolic groups and groups that are virtually torsion free. We know
that all finitely generated virtually free groups have a proper action
(on a single bounded valence tree), so the next
family of groups we would expect to consider are the fundamental groups
$S_g$ of the closed orientable surfaces $\Sigma_g$ for genus $g\geq 2$.
It is straightforward to show that these groups act on a product
of two trees with finite stabilisers if we allow one of these trees
to have infinite valence:

\begin{ex} \label{srf2} 
Taking $g=2$ (as $S_g$ is a subgroup of $S_2$ for $g\geq 2$), we have that
$S_2$ acts on its Bass - Serre tree $T_1$ via
the amalgamation $F_2*_{\Z} F_2$ for $\langle a,b\rangle$ the first $F_2$
factor and $\langle c,d\rangle$ the second, where we have $[a,b][c,d]=id$.
Although $T_1$ is of course
not locally finite, the elliptic elements of this action are all
conjugate into the vertex groups $\langle a,b\rangle$ or $\langle c,d\rangle$.
Now consider the homomorphism from $S_2$ onto the rank 2 free group
$\langle x,y\rangle$ given by $a,d\mapsto x$ and $b,c\mapsto y$. As
$\langle x,y\rangle$ acts freely and purely loxodromically on a locally
finite tree $T_2$, we have that our surface group $S_2$ acts on $T_2$ where
any non identity element in $\langle a,b\rangle$ or $\langle c,d\rangle$
is loxodromic. This is also true for any non identity element conjugate in
$S_2$ into $\langle a,b\rangle$ or $\langle c,d\rangle$ so by Proposition
\ref{hyel} we have that $S_2$ (and $S_g$ for $g\geq 2$) acts properly 
preserving factors on a product of two trees $T_1\times T_2$,
where the second is locally finite but the first is not.
\end{ex}

In fact these surface groups and many more groups besides are
known to act metrically properly on a finite product of
infinite valence trees. In \cite{drnjan} it was shown that 
Coxeter groups have this property by utilising their action on the Davis
complex. Thus any group that virtually embeds in a Coxeter group (which
includes any group virtually embedding in a RAAG) will also have this
property. However the question of whether such an action exists when all
factor trees are locally finite was raised in \cite{flss} and seems an
important and difficult question which we will consider. As an indication
of why this might be so, note that in Example \ref{srf2} our action on
the locally infinite tree has finitely generated vertex and edge groups,
but this is never true for $S_g$ on a locally finite tree.
\begin{ex} In any splitting of $S_g$
  with respect to an action on a bounded valance (or locally finite)
  tree, all edge and vertex groups are infinitely generated. To see this,
  note that local finiteness implies that all these subgroups are commensurable
  to each other in $S_g$, so if one is finitely
  generated then they all are. Moreover any vertex group is commensurable
to any of its conjugates. But $S_g$ is a torsion free hyperbolic 
group where any finitely generated subgroup is quasi-convex. In
\cite{cnmh} Theorem 3.15
it is shown that an infinite quasi-convex subgroup of a hyperbolic
group has finite index in its commensurator. Hence
if such an action on a bounded valence tree existed with finitely generated
vertex/edge stabilisers then these would either have finite index in $G$,
so that actually the splitting is trivial, or the stabilisers are themselves
trivial in which case $S_g$ would be free.
\end{ex}

In order to say something about possible proper actions of $S_g$ on products
of locally finite trees, we first consider torsion free hyperbolic groups
or more generally groups $G$ which are finitely
generated and torsion free but which do not contain $\Z\times\Z$.
We obtain an obstruction for such an action, which is not unlike
\cite{flss} Theorem 15.

\begin{thm} \label{injcloc}
  Suppose the torsion free and finitely generated group $G$ acts properly
  on the product of $k$ locally finite
trees $P=T_1\times\ldots\times T_k$ preserving factors
but does not contain $\Z\times\Z$. Then 
for each non identity element $g\in G$ there is an action of $G$ on some
bounded valence tree $T$ which is minimal, faithful and
such that $g$ acts loxodromically.
\end{thm}
\begin{proof}
We first remove any tree where the
projection action of $G$ has a global fixed point, but having done that
we also remove in turn any tree $T_i$ where the action of $G$ on $T_i$ has the
following property: for any element $g\in G$ which is loxodromic in this
action, there is some other tree left in the product (maybe depending on $g$) 
where $g$ also acts loxodromically. On removing this tree $T_i$ from the
product, we have by Proposition \ref{hyel} that $G$ is still acting
properly and we continue until no such tree is left, whereupon we revert
to the original notation $T_1\times\ldots\times  T_k$ for the product
of trees thus obtained.

We now replace each of these remaining trees $T_i$ by its core $C_G(T_i)$,
though we will continue to call it $T_i$. This
has bounded valence and $G$ acts minimally on it. Moreover
the action on the ensuing product of trees 
$T_1\times\ldots\times T_k$ is still proper
by Proposition \ref{hyel}, because the process of restricting an action
to the core does not affect whether an element is loxodromic or elliptic.

Now suppose that the action of
$G$ on $T_1$ is not faithful, so there exists an infinite order element
$w\in G$ acting trivially on $T_1$. But there will exist other elements of
$G$ acting loxodromically on $T_1$ or else it would have been removed in the
first stage above. Furthermore there is some element $h\in G$ acting
loxodromically on $T_1$ and such that the action of $h$ is elliptic on all
of $T_2,\ldots ,T_k$ (or else $T_1$ would have been removed at some point
during the second stage). This means that no positive power $h^n$ can
commute with $w$ in $G$, as otherwise $\langle h^n,w\rangle\cong\Z\times\Z$
unless $h^{rn}=w^s$ for some non zero $r,s$ but this would imply that the
action of $h$ on $T_1$ has finite order.

Hence on $T_2,\ldots ,T_k$ we have that $h$ is elliptic, as is $whw^{-1}$.
We now use local finiteness of our trees: observe that
for each $2\leq j\leq k$ there is $n_j$ such that $h^{n_j}$ and
$wh^{n_j}w^{-1}$ have a common fixed point, because if $h$ fixes the vertex
$v$ then some power of $h$ fixes all edges incident at $v$ and then a further
power will fix the vertices at distance 2 away from $v$ and so on.
Consequently
$wh^{n_j}w^{-1}h^{-n_j}$ is also elliptic in the action on $T_j$. But this
argument also applies to any multiple of $n_j$, thus for $N=n_2\ldots n_k$
we have that $wh^Nw^{-1}h^{-N}$ is elliptic on $T_2,\ldots ,T_k$. Moreover
on $T_1$ this element acts as the identity because $w$ does, but we noted above
that it is not the identity element in $G$. So the action of $G$ on
$T_1\times\ldots\times T_k$ is not proper by Proposition \ref{hyel}.

Thus all projection actions are faithful and minimal. Moreover given any
non identity element $g$ in $G$, we can take the action of $G$ on one of
the factor trees $T$ where $g$ acts loxodromically.
\end{proof}
Note: The point of this theorem is that it is significantly stronger
to ask that the action is both faithful and minimal. Indeed if we just
require a minimal action, as the surface group $S_2$ (and $S_g$ for higher
genus $g$) is a residually free group, we can pick any non identity
$g\in S_2$ and take a homomorphism from $S_2$ to a free group $F_r$
which does not vanish on $g$. We can then let $F_r$
act on its Cayley graph, thus $S_2$ acts on this graph too with $g$ a
loxodromic element and this is a minimal action. Alternatively if we want
a faithful action then
we can now ``decorate'' this action of $S_2$ to convert it
into a faithful action of $S_2$ on a bigger (but still locally finite)
tree, though this action will certainly not be minimal.

In the next section we will show that this does indeed hold for the groups
$S_g$. However
we finish this section by noting that things are very different for groups 
containing 
$\Z\times\Z$. Taking $G=F_2\times\Z$ with $\Z=\langle z\rangle$, we have that
$G$ acts properly (and even cocompactly) on a product of two locally
finite trees and is also residually free. But there is no faithful
minimal action of $G$ on any tree whatsoever. This is because if $G$ acts
on a tree with $z$ a loxodromic element then any $g\in G$ sends the axis
of $z$ to itself, so the core must be the axis of $z$ and then the 
minimal action
is not faithful. If however $z$ acts elliptically then for all loxodromic
elements $g$, we have that $z$ sends the axis of $g$ to itself and
preserves the direction (else it conjugates $g$ to its inverse), so
$z$ must fix this axis pointwise and hence fix the whole core pointwise.

\section[Fields of positive characteristic]{Fields of positive 
characteristic and the Bruhat - Tits tree}

Question 1 of \cite{flss} asks: let $S_g$ be a closed surface group of
genus $g\geq 2$. Is there a discrete and faithful representation of
$S_g$ into $Aut(Y)$ for $Y$ a finite product of bounded valence
trees? 
As $S_g$ is torsion free and all trees considered here are
locally finite, a discrete and
faithful representation of $S_g$ is the same as a proper action. 
Moreover Lemma \ref{equiv} says it does not matter whether we use
locally finite or bounded valance trees.
Thus the question is equivalent to asking whether $S_g$ acts properly on a 
finite product of locally finite trees, and hence equivalent to whether
$S_2$ does, by induced actions. (This question also appears
in \cite{wshlv} as Problem 10.13 for the case of a product of two locally
finite trees.)
It is pointed out in Theorem 3 of \cite{flss} that if we can find a
faithful representation of a finitely generated group $G$ into
$PGL(2,K)$ for $K$ a global field of characteristic $p>0$, say
$\F_p(x)$, then $G$ acts properly on a finite product of locally
finite trees. This is because for each valuation $v$ of $K$, the group
$PGL(2,K)$ will act faithfully on its Bruhat - Tits tree, for instance
the regular tree $T_{p+1}$ when $K=\F_p(x)$. Although this will not
be a proper action in general, the finite generation of $G$ means we
can take finitely many valuations on $K$ to get a proper action on the product
of these trees. (If $K$ is a global field of characteristic zero, namely
a number field, the above still works except that those elements of $G$
whose trace is an algebraic integer will be elliptic in every action.)

In the case where $G=S_g$, the existence of such a representation
seems a hard question, but on extending the field $\F_p(x)$ by one
transcendental element we can ask whether we have an embedding of $S_g$
using this new field. Indeed Theorems 4 and 5 in \cite{flss} 
show that for every prime $p$ at least 5, there is a faithful embedding
of $S_2$ in $PGL(2,K)$ where $K$ is a finite extension of $\F_p(x,y)$
and for any characteristic $p$ field $k$ of transcendence degree at least
2, there is a faithful embedding of $S_2$ in $PGL(n,k)$ for some $n$.
Here we will improve on these results by giving a completely explicit
faithful embedding of $S_2$ in $SL(2,K)$, and hence in $PSL(2,K)$
for $K=\F_p(x,y)$ where $p$ is any prime. To obtain faithfulness
of the representation, we use the following result of Shalen.

\begin{prop} (\cite{sha} Proposition 1.3) \label{shl}
Suppose $G_1*_HG_2$ is a free product with abelian amalgamated
subgroup $H$ and suppose we have faithful representations
$\rho_i:G_i\injects SL(2,\F)$ and
$i=1,2$ over any field $\F$ such that\\
(a) $\rho_1$ and $\rho_2$ agree on $H$,\\
(b) $\rho_1(h)=\rho_2(h)$ is diagonal for all $h\in H$ and\\
(c) For all $g\in G_1\setminus H$ we have that the bottom left hand entry
of $\rho_1(g)$ is non zero, and similarly for the top right hand entry
of $\rho_2(g)$ for all $g\in G_2\setminus H$.

Then $G_1*_HG_2$ embeds in $SL(2,\F(y))$ where $y$ is a transcendental
element over $\F$. 
\end{prop}
\begin{proof} In \cite{sha} the result was stated just for the field
$\C$ but for arbitrary dimension $d$. However the proof does work for
arbitrary fields $\F$ and general dimensions. Here we just give a summary
in the $d=2$ case, including one point in the proof which will be needed
later.

Define the representation $\rho:G_1*_HG_2\rightarrow SL(2,\F(y))$ as
equal to $\rho_2$ on $G_2$ but on $G_1$ we replace $\rho_1$ by the
conjugate representation $T\rho_1T^{-1}$ where $T$ is the diagonal
matrix $\mbox{diag}(1,y)$, and then extend to all of
$G_1*_HG_2$. Now it can be shown straightforwardly 
that any element not conjugate into $G_1\cup G_2$ is
conjugate in $G_1*_HG_2$ to something with normal form
\[g=\gamma_1\delta_1\ldots \gamma_l\delta_l\]
where all $\gamma_i\in G_1\setminus H$ and all $\delta_i\in G_2\setminus H$.
Induction on $l$ then yields that the entries of $g$ are Laurent polynomials
in $y^{\pm 1}$ with coefficients in $\F$ and with the bottom right 
hand entry of $g$ equal to $\alpha y^l+\ldots$ 
where all other terms are of
strictly lower degree in $y$. But it can be checked
that $\alpha$ is actually just a product of these respective
bottom left and top right entries, thus is a non zero element of $\F$
so this bottom right hand entry does not equal 1 and
$g$ is not the identity matrix.

Moreover the top left hand entry of $g$ is equal to a Laurent polynomial of
the form $\beta y^{l-1}$ plus lower order terms, for $\beta\in\F$ (although
unlike $\alpha$ above, $\beta$ could be zero). This means that the trace of
$g$ is also of the form $\alpha y^l$ plus lower order terms. 
\end{proof}   
  
We can use this result as follows:

\begin{co} \label{shas}
Let $\F$ be the field $\F_p(x)$. Suppose we have a pair of
2 by 2 matrices $A,B\in SL(2,\F)$ such that $\langle A,B\rangle$ is a
free group of rank 2 and $ABA^{-1}B^{-1}$ is a diagonal matrix. Then on 
introducing a transcendental element $y$ and setting $D=TAT^{-1}, C=TBT^{-1}$
for $T=\mbox{diag}(1,y)$, we have that $\langle A,B,C,D\rangle$ is
a faithful representation of the surface group $S_2$ in $SL(2,\F_p(x,y))$.
\end{co}
\begin{proof}
This is the case in the above proposition where $G_1=\langle a,b\rangle$ and 
$G_2=\langle d,c\rangle$ are both copies of the free group $F_2$, 
where $a=d$ and $b=c$. On setting $H$ to be the cyclic subgroup
generated by $aba^{-1}b^{-1}=dcd^{-1}c^{-1}$, we see that $G_1*_H G_2$ is 
$S_2$. We let $\rho_1$ send $a,b$ to $A,B$ and $\rho_2$ send $d,c$ to $A,B$,
thus (a) and (b) are satisfied. But (c) is satisfied too, because
an element $X$ of $SL(2,\F)$ with an off diagonal entry zero would have the
property that $\langle X,ABA^{-1}B^{-1}\rangle$ is a soluble group, which
which cannot happen in the free group on $A,B$ if
$X\notin \langle ABA^{-1}B^{-1}\rangle$.
\end{proof}

We now look for matrices of the required form.
\begin{thm} \label{matfrm}
Let $\F$ be any infinite field and let $c,h,d,\delta$ be non zero
elements of $\F$. Set $X=1-d\delta h+d^2 h^2$, $Y=\delta^2-d \delta h+h^2$
and suppose that $X$ and $Y$ are also non zero. Then on defining
\[
A=\sma{cc}\frac{dY}{X}&\frac{d\delta h(1-d^2)+d^2\delta^2-1}{cX}\\
c&d\fma
\mbox{ and }B=\sma{cc}\frac{\delta X}{Y}&
\frac{d\delta(1-\delta^2)+h(d^2\delta^2-1)}{cY}\\
ch&\delta\fma,\]
we have that $A,B\in SL(2,\F)$ with
$\mbox{tr}(A)=d(X+Y)/X$, $\mbox{tr}(B)=\delta (X+Y)/Y$
and $\mbox{tr}(AB)=(d\delta(1+h^2)-h)(X+Y)/(XY)$. We also have
\[AB=\sma{cc}\frac{d\delta(1+h^2)-h}{X}&
\frac{dh(\delta^2-1)+\delta(d^2-1)}{cX}\\
\frac{c(\delta+dh^3)}{Y}&\frac{d\delta(1+h^2)-h}{Y}
\fma,
BA=\sma{cc}\frac{d\delta(1+h^2)-h}{Y}&
\frac{dh(\delta^2-1)+\delta(d^2-1)}{cY}\\
\frac{c(\delta+dh^3)}{X}&\frac{d\delta(1+h^2)-h}{X}\fma\]
\[\mbox{ and }ABA^{-1}B^{-1}=\sma{ccc}\frac{Y}{X}&0\\0&\frac{X}{Y}\fma.\]
\end{thm}
\begin{proof} Although the form of these matrices was found by working
  backwards in two stages by first solving for $AB$ and $BA$, then $A$
  and $B$, we can confirm this by direct calculation (by computer if
  preferred).
\end{proof}

It remains to be seen that we can find matrices $A,B$ of the above form
which generate a free group of rank 2 when $\F=\F_p(x)$.
To do this, we can take the 
discrete valuation 
$v$ on $\F_p(x)$ given by minus the degree, so 
$(a_m x^m+\ldots +a_0)/(b_n x^n+\ldots +b_0)$ has valuation $n-m$ if
$a_m,b_n\neq 0$. 
Recall that a discrete valuation $v:\F\rightarrow
\Z\cup\{\infty\}$ on a field $\F$ satisfies:\\
(1) $v(x)=\infty$ if and only if $x=0$\\
(2) $v(xy)=v(x)+v(y)$\\
(3) $v(x+y)\geq \mbox{min}(v(x),v(y))$. Moreover if $v(x)\neq v(y)$
then this is an equality.\\
The set of elements ${\mathcal O}_v=\{x\in \F:v(x)\geq 0\}$ forms a
subring of $\F$, called the valuation ring, which is a principal ideal
domain and an element $\pi$ with $v(\pi)=1$ is called a uniformiser.

We can then take the metric completion $k$ of
$\F_p(x)$ to obtain a local field, with the same valuation
and which also acts on its Bass - Serre tree
$T_{p+1}$. (The above is evaluation at zero: perhaps the more common
valuation used is that at infinity, giving $k=\F_p((x))$ but this is
obtained anyway by substituting $1/x$ for $x$ in everything below.)  
The results of \cite{con} then tell us when a pair of matrices
$A,B\in SL(2,k)$ generate a free and discrete group. A necessary condition
is that the valuation $v$ of the traces $A,B$ and $AB$ are all negative, as
otherwise these will act on $T_{p+1}$ as elliptic elements. Otherwise the
translation length of a matrix $M\in SL(2,k)$ is $-2v(\mbox{tr}\,M)$.
Now if we can find matrices $A,B$ in the above form where
the valuation of the three traces $\mbox{tr}(A),\mbox{tr}(B),\mbox{tr}(AB)$
are all equal and negative (say $-1$) then we are in Case 2(i) of
Proposition 3.5 in \cite{con}, with this satisfying the hypothesis in
Corollary 3.6 which shows that $\langle A,B\rangle$ is free of rank 2 and
discrete.
\begin{thm} \label{mat}
If $p$ is any odd prime then the following matrices 
$A,B,C,D$ in 
$SL(2,\F_p(x,y))$ generate a faithful representation of the genus 2 surface
group, where we have $ABA^{-1}B^{-1}=DCD^{-1}C^{-1}$.
\begin{eqnarray*}
A=\sma{cc} \frac{1-2x^2-2x^3}{x(x-1)}&\frac{1-2x^2-x^3-x^4}{x^3(x-1)}\\
1&1/x^2\fma &,&
B=\sma{cc} \frac{x^2-1}{x-2x^3-2x^4}&
\frac{1+2x-x^2-3x^3-2x^4}{x^2(2x^3+2x^2-1)}\\x&1+x\fma,\\
D=\sma{cc} \frac{1-2x^2-2x^3}{x(x-1)}&\frac{1-2x^2-x^3-x^4}{yx^3(x-1)}\\
y&1/x^2\fma &,&
C=\sma{cc} \frac{x^2-1}{x-2x^3-2x^4}&
\frac{1+2x-x^2-3x^3-2x^4}{yx^2(2x^3+2x^2-1)}\\yx&1+x\fma.
\end{eqnarray*}
\end{thm}
\begin{proof}
The form of $A$ and $B$ come from the previous
result with ($c=1$ without loss of generality and)
$d=1/x^2$, $\delta=x+1$ and $h=x$, whereupon
we do find (for $p\neq 2$) that the traces of $A,B,AB$ all have valuation
$-1$. Hence $\langle A,B\rangle$ is a rank 2 free group in the required
form for the application of Corollary \ref{shas}.

Although not needed for this proof, we briefly
indicate how $d,\delta,h$ were chosen. Looking at the matrices in Theorem
\ref{mat},
a necessary condition for $\langle A,B\rangle$ to be discrete and free is that
$ABA^{-1}B^{-1}$ is loxodromic, so we require $v(Y)\neq v(X)$. On picking 
$v(X)=1$ and $v(Y)=-2$ (a somewhat ad hoc choice, obtained by examining a
specific case that was found to work by computation), if we want the valuation
of $\mbox{tr}(A)=d(X+Y)/X$ to be $-1$ under our choices for $v(X)$ and $v(Y)$
then this happens if and only if $v(d)=2$. Similarly the
valuation of $\mbox{tr}(B)=\delta(X+Y)/Y$ being $-1$ under the same condition
is equivalent to $v(\delta)=-1$.
If we now try to satisfy $v(X)=v(1-d\delta h+d^2 h^2)=1$ with $v(d)=2$ and
$v(\delta)=-1$ then we cannot take $v(h)\geq 0$ because this implies that
$v(X)=0$. However $v(h)=-1$ would work
if there is some cancellation in forming $1-d\delta h$, each term of which has
valuation 0 but whose difference we will now assume has valuation $1$. 
These values
for $v(d),v(\delta),v(h)$ also give $v(Y)=-2$ as required if there is no  
cancellation when adding $\delta^2$ and $h^2$, both of which have valuation
$-2$.

As for obtaining $v(\mbox{tr}(AB))=-1$, this now happens if 
$v(d\delta+h(d\delta h-1))$ is 0. As we have assumed above one case of
cancellation in forming $d\delta h-1$, this will be true.

To ensure these conditions hold for specific choices of $d,\delta,h$, we aim
for the simplest expressions we can find. If $d\delta h=(x+1)/x$ then we
will have $v(d\delta h)=0$ but $v(d\delta h-1)=v(1/x)=1$ which gives us
the required cancellation. Thus on choosing
$\delta=x+1, h=x$ and therefore $d=1/x^2$, we see that $\delta^2+h^2=
2x^2+2x+1$, so if $p\neq 2$ then we do not have cancellation in forming this
sum and so all the required conditions above are satisfied.

Unfortunately this argument cannot work in $\F_2(x)$ because there will
always be cancellation when adding elements with the same valuation.
To deal with this case, we tried further possibilities for the valuations,
this time looking for the traces of $A,B,AB$ each to have valuation $-2$. 
On trying $v(d)=4,v(\delta)=-2,v(h)=-2$, we see that this would imply
$v(X)>0$ (because of the cancellation between 1 and $d\delta h$),
whereas $v(Y)>-4$ (because of cancellation between $\delta^2$ and $h^2$).
If we could arrange to have lots of cancellation between $d\delta h$ and 1
so that $v(d\delta h+1)=5$, giving
$v(X)=4$, and some cancellation between $\delta^2$ and $h^2$ giving
$v(\delta^2+h^2)=-2$ and thus $v(Y)=-2$,
then we obtain suitable traces for $A,B$ and $v(\mbox{tr}(AB))$ will also
be $-2$ if $v(Z)=1$, which will hold from imposing $v(d\delta h+1)=5$ above.

Thus we trying setting $\delta=x^2$ and $h=x^2+x+1$ to give us the
correct amount of cancellation in $\delta^2+h^2$. Then $d\delta h$ has to be 
something like $x^5/(x^5+1)$ to provide enough cancellation when adding it
to 1. Thus we also set $d=x^3/((x^2+x+1)(x^5+1))$ and these values satisfy
all of the above conditions, hence we have shown:
\end{proof}
\begin{co} \label{mat2}
The following matrices 
$A,B,C,D$ in 
$SL(2,\F_2(x,y))$ generate a faithful representation of the genus 2 surface
group, where we have $ABA^{-1}B^{-1}=DCD^{-1}C^{-1}$.
\begin{eqnarray*}
A&=&\sma{cc} 
\frac{x^8 + x^7 + x^5 + x^4 + x^3}{x^6 + x^5 + 1}&   
\frac{x^{13} + x^{11} + x^2 + x + 
    1}{(x^6 + x^5 + 1) (x^5 +1) (x^2 + x + 1)}\\
1&\frac{x^3}{(x^5 +1) (x^2 + x + 1)}\fma,\\
B&=&\sma{cc}
\frac{x^8 + x^7 + x^2}{(x^7+x^2+1) (x^5 + 1)}&\frac{x^{12} + x^{10} + x^9 + 
x^5 + x^4 + x^2+ 1}{(x^7 + x^2 + 1)(x^5 +1) (x^2 + x + 1)}\\
x^2 + x + 1&x^2\fma,
\end{eqnarray*}
and $D,C$ are the conjugates of $A,B$ respectively by $\mbox{diag}(1,y)$.
\end{co}

We now show that the necessary condition for a proper action on a finite 
product of locally finite trees, given in the previous section, is
satisfied by the surface group $S_g$.

\begin{co} \label{moreacc}
If $S_g$ is the closed orientable hyperbolic surface group of genus $g\geq 2$ 
then for any finite collection of non-identity elements
$\gamma_1,\ldots ,\gamma_m\in S_g$ there is an action of $S_g$
on some locally finite tree which is minimal, faithful and such that each of
$\gamma_1,\ldots,\gamma_m$ act loxodromically.
\end{co}
\begin{proof}
We show this for the group $S_2$ as all other $S_g$ are finite index
subgroups of $S_2$, so the action will still be minimal. We first suppose
that we are given one non-identity element $\gamma$ in $S_2$.

For any prime $p$,
we regard the global field $\F_p(x)$ as sitting inside its metric completion,
the local field $k$, whereupon $SL(2,k)$ also acts on the $p+1$ regular
tree by automorphisms, and does so faithfully apart from $-I$.
Given the above matrices $A,B,C,D\in SL(\F_p(x,y))$, we can regard
them as lying in $SL(2,k)$ on taking $y$ to be any element in $k$ which is
transcendental over $\F_p(x)$ and such elements exist by countability
considerations. We thus obtain a faithful action of $S_2$ on the tree
$T_{p+1}$. Now the free subgroups $\langle A,B\rangle$ and $\langle C,D\rangle$
both act purely loxodromically by construction, so our element $\gamma$
will automatically be a loxodromic element if it is conjugate into either of
these subgroups. Otherwise, as mentioned at the end of the proof of
Proposition \ref{shl}, the trace of $\gamma$
will be a Laurent polynomial in the variable $y$ of the form 
$\alpha y^l$ plus lower order terms in $y$,
where $l>0$ and
$\alpha$ is a non zero element of $\F_p(x)$. Now this expression does not 
change if we change the element $y$ in $k$, as long as it is still 
transcendental over $\F_p(x)$.
To ensure that $\gamma$ is loxodromic here, we require that its trace
has negative valuation. But if $y$ is a transcendental element of $k$ then
for $n\in\Z$ we have that $z=yx^n\in k$ is still transcendental over
$\F_p(x)$ and with $v(z)=v(y)-n$. Thus regardless of $v(\alpha)$ or the
other coefficients in the Laurent polynomial for $\mbox{tr}(\gamma)$, if
we take $n$ large enough and replace $y$ with $z$ then this trace will have
negative valuation. Now suppose that we are given finitely many non-identity
elements $\gamma_1,\ldots ,\gamma_m$ of $S_2$. The above argument works
with $\alpha$ and $l$ depending on the element $\gamma_i$, but we can simply
take $n$ large enough so that all of these finitely many traces have negative
valuation.

Finally we are not claiming that this action will definitely be minimal. But
if we restrict to its core, we have a minimal action in
which loxodromic elements remain loxodromic. Thus we can only lose
faithfulness of the action if there were some non identity element in $S_2$
which is acting elliptically on $T_{p+1}$ in the original action but
which acts as the identity when restricted to the core. However this action
of $SL(2,k)$ extends to an action on the boundary of the tree $T_{p+1}$, 
which is projective space $\mathbb P^1(k)$, where elements act as
M\"obius transformations. Such a transformation fixes
at most two ends, whereas if this element were acting as the identity on
the core, it would fix all ends of the core. But $A$ and $B$ provide two
independent loxodromic elements in this action, thus the core has more than
two ends and hence this restriction is still faithful.
\end{proof}

We finish by noting that the question of whether
the surface group $S_g$ can act properly on a
finite product of locally finite trees
could also be asked for any hyperbolic group $H$ that is not virtually
free. However these two cases are very much connected, on
recalling the famous open question credited to Gromov asking 
whether such a group always contains a surface subgroup. If we could find
such a group $H$ exhibiting a positive answer to our more general
question then either $H$ contains a surface subgroup $S_g$ for some
$g\geq 2$, thus $S_g$ also acts properly on a finite product of locally finite
trees for all $g\geq 2$, or we would have a counterexample to Gromov's
question. As for non-hyperbolic groups, it would also be interesting to ask
which (finitely generated) RAAGs have a proper action
on a finite product of locally finite trees. The only answer we know of
here is that direct products of free groups do (and indeed they do not
contain surface groups $S_g$).

\bibliographystyle{plain}

\end{document}